\documentclass[a4paper,11pt]{article}
\usepackage{pos}

 \newtheorem{proposition}{Proposition}[section]
  \newtheorem{lemma}[proposition]{Lemma} 
  \newtheorem{corollary}[proposition]{Corollary}

\theoremstyle{definition} 
  \newtheorem{definition}[proposition]{Definition}
    
  \newtheorem{example}[proposition]{Example}

  \theoremstyle{remark} 
  \newtheorem{remark}[proposition]{Remark}

  \newcounter{c} 
   
  \newcommand{\etyk}[1]{\vspace{-7.4mm}$$\begin{equation}\Label{#1} 
  \addtocounter{c}{1}} 
  \renewcommand{\]}{\ifnum \value{c}=1 $$\else \end{equation}\fi} 
   \numberwithin{equation}{section}

\def\lto{\longmapsto}
\def\lra{\longrightarrow}
\def\FF{{\mathbb F}}
\def\ZZ{{\mathbb Z}}
\def\lb{\pmb{\{}}
\def\rb{\pmb{\}}}
\def\blb{\pmb{\big\{}}
\def\brb{\pmb{\big\}}}
\def\Der{\mathrm{Der}}

\title{Lie trusses and heaps of Lie affebras}

\author*[a,b]{Tomasz Brzezi\'nski}

\affiliation[a]{Department of Mathematics, Swansea University \\
Swansea University Bay Campus,
Fabian Way,
Swansea,
  Swansea SA1 8EN, U.K.}

\emailAdd{T.Brzezinski@swansea.ac.uk}

\affiliation[b]{Faculty of Mathematics, University of Bia{\l}ystok\\
 K.\ Cio{\l}kowskiego  1M,
15-245 Bia\-{\l}ys\-tok, Poland}

\abstract{A frame-independent formulation of Lie brackets on affine spaces or {\em Lie affebras} introduced  in [K.\ Grabowska, J.\ Grabowski \& P.\ Urba\'nski,  Lie brackets on affine bundles, {\em Ann.\ Global Anal.\ Geom.} {\bf 24} (2003), 101--130] is given.}

\FullConference{%
  Corfu Summer Institute 2021 "School and Workshops on Elementary Particle Physics and Gravity"\\
  29 August - 9 October 2021\\
  Corfu, Greece
}

\begin{document}
\maketitle

\section{Introduction}
A gauge or frame-independent formulation of analytical dynamics \cite{Wei:uni}, \cite{Tul:fra}, \cite{Ben:fib}, \cite{Urb:aff}, \cite{MasVig:non} requires one to replace vector bundles by affine fibre bundles (see \cite[Section~1]{GraGra:AV} for an excellent summary of reasons for that). Motivated by this K.\ Grabowska, J.\ Grabowski and P.\ Urba\'nski introduced and studied Lie brackets on affine bundles (or on affine spaces) in a series of papers \cite{GraGra:Lie}, \cite{GraGra:AV}, \cite{GraGra:fra}. They coined the term {\em Lie affebra} to describe an affine space together with an anti-symmetric, bi-affine bracket with values in the underlying vector space that satisfies a suitable version of the Jacobi identity. Alas, the usage of a vector space in this definition seems to indicate a departure from the frame- or observer-independent philosophy as every vector space has a distinguished element (the zero vector). In the present note we propose a reformulation of Lie affebras in the way in which no distinction between elements of the affine space is required. The key here is the observation that an affine space (and affine transformation) can be defined without reference to a vector space, as a set with two ternary operations. The first of these reflects the translation of a point by a unique vector defined by two other points, and describes an algebraic system known as an {\em abelian heap} \cite{Pru:the}. The second operation represents the action of the field of scalars and is obtained as translation of a point by a rescaled vector originating at this point. This leads us to define a Lie bracket on an abelian heap (in particular an affine space) as a ternary operation satisfying compatibility with ternary heap operation (or both ternary operations that define an affine space). These conditions include a suitable anti-symmetry property and the Jacobi identity. The latter is necessarily different from that discussed in context of Jacobson's  triple Lie systems \cite{Jac:Lie} or  Filippov's $3$-Lie algebras \cite{Fil:Lie} as no universally defined addition of points of an abelian heap (or an affine space) is available. We term an abelian heap with a ternary Lie bracket a {\em Lie truss} and an affine space with a ternary Lie bracket a {\em heap of Lie affebras}. We show that through fixing a point in an affine space and thus deriving its underlying vector space heaps of Lie affebras can be reduced to Lie affebras of Grabowska, Grabowski and Urba\'nski and, consequently, Lie algebras. In a similar way fixing an element of a Lie truss one obtains a Lie ring.

The presented {\em internalisation} of the definition of a Lie affebra leading to systems of interacting ternary operations seems to mesh well with the recent resurgence of interest in such systems, both in pure mathematics (see e.g.\ \cite{DupWer:str} and references therein) and in mathematical physics (see e.g.\ \cite{Bru:sem},  \cite{Ker:ter0},\cite{Ker:ter}).

\section{Heaps, affine spaces and trusses}
As is well-known (see e.g.\ \cite{ncat:aff}) but explored most recently in \cite{BreBrz:hea} an affine space can be defined as a set with two ternary operations. First, recall that an {\em abelian heap} \cite{Pru:the}, \cite{Bae:ein} is a set $A$ with a ternary operation
$$
[-,-,-]: A \times A\times A\lra A,
$$
such that, for all $a_i \in A$, $i=1,\ldots, 5$,
\begin{subequations} \label{heap}
\begin{equation}\label{h.assoc}
[a_1,a_2, [ a_3,a_4,a_5]] = [[a_1,a_2,  a_3],a_4,a_5]
\end{equation}
\begin{equation}\label{h.Mal}
[a_1,a_2, a_2] = a_1 =  [a_2,a_2,  a_1] 
\end{equation}
\begin{equation}\label{h.ab}
[a_1,a_2, a_3] = [a_3,a_2,  a_1] 
\end{equation}
\end{subequations}
A {\em homomorphism of abelian heaps} is a function $f: A\lra B$, such that
\begin{equation}\label{h.hom}
f([a_1,a_2, a_3]) = [f(a_1),f(a_2), f(a_3)], \qquad \forall a_i\in A.
\end{equation}
With every abelian heap $A$ one can associate a family of mutually isomorphic abelian groups labelled by elements of $A$ as follows. For all $o\in A$, $A$ is a group with neutral element $o$ and addition and inverses defined by
\begin{equation}\label{retract}
a+b = [a,o,b], \qquad -a = [o,a,o].
\end{equation}
This group is called the {\em retract of $A$ at $o$} and is denoted by $A(o)$. Conversely, any abelian group $A$ can be understood as an abelian heap with the operation 
\begin{equation}\label{g.h}
[a,b,c] = a-b+c.
\end{equation}

The rules of an abelian heap \eqref{heap} imply that the placing of brackets in-between any odd number of elements of $A$ does not change the value of the expression, and hence we will write
$[a_1,a_2,\ldots, a_{2n+1}] \in A$ for any possible repeated application of $[-,-,-]$ to $(a_1,a_2,\dots,a_{2n+1})$. Furthermore, any pair of identical neighbours cancel each other out (we refer to this as the {\em cancellation rule}), and 
$$
{}[a_1,a_2,\ldots, a_{2n+1}]  = [a_{\sigma(1)},a_{\zeta(2)}, \ldots, a_{\zeta(2n)}, a_{\sigma(2n+1)}],
$$
where $\sigma$ is any permutation of $\{1,3,\ldots, 2n+1\}$ and $\zeta$ is a permutation of $\{2,4,\ldots, 2n\}$. This is referred to as the {\em reshuffling rule}. Thus any pair of identical elements occupying positions of opposite parity can be safely removed and provided that the remaining elements are reshuffled so that the parity of their final positions is identical with the parity of their original positions the value of the expression in brackets will be retained.

An affine space over a field $\FF$  can be identified with a non-empty unital and absorbing {\em heap of modules} over $\FF$ \cite{BreBrz:hea}, i.e.\ with an abelian heap $A$ together with an operation 
$$
\Lambda: \FF\times A\times A\lra A,
$$ 
satisfying the following conditions, for all $\alpha,\beta,\gamma \in \FF$ and $a,b,c,d\in A$,
\begin{subequations}\label{a.space}
\begin{equation}\label{l.heap}
\Lambda(\alpha -\beta +\gamma , a,b) = [\Lambda(\alpha,a,b), \Lambda(\beta,a,b), \Lambda(\gamma,a,b)],
\end{equation}
\begin{equation}\label{r.heap}
\Lambda(\alpha, a,[b,c,d]) = [\Lambda(\alpha,a,b), \Lambda(\alpha,a,c), \Lambda(\alpha,a,d)],
\end{equation}
\begin{equation}\label{a.assoc}
\Lambda(\alpha\beta, a,b) = \Lambda(\alpha, a, \Lambda(\beta, a,b)),
\end{equation}
\begin{equation}\label{b.change}
\Lambda(\alpha, a,b) = [\Lambda(\alpha, c,b),\Lambda(\alpha, c,a),a],
\end{equation}
\begin{equation}\label{0.1}
\Lambda(0,a,b) = a = \Lambda(1,b,a).
\end{equation}
\end{subequations}
From this perspective, a homomorphism of affine spaces is a homomorphism of heaps $f:A\lra B$, such that, for all $a,b\in A$ and $\alpha\in \FF$,
\begin{equation}\label{hom.aff}
f\left(\Lambda_A(\alpha, a,b)\right) = \Lambda_B(\alpha, f(a), f(b)).
\end{equation}

A connection with perhaps more familiar formulation of the definition of an affine space as a set $A$ with a free and transitive action of an $\FF$-vector space $\overset{\lra}A$ is given in the following way. Let $\overset{\lra}{ab}$ denote the unique vector in $\overset{\lra}A$ such that $b= a+ \overset{\lra}{ab}$. Then, for all $a,b,c\in A$,
\begin{equation}\label{heap.aff}
[a,b,c] = a+\overset{\lra}{bc}, \qquad \Lambda(\alpha, a,b) = a+ \alpha\,\overset{\lra}{ab}.
\end{equation}

A few comments about conditions \eqref{a.space} are now in order. Equation \eqref{l.heap} means that, for all $a,b\in A$, $\Lambda(-,a,b): \FF\lra A$ is a homomorphism of heaps (where $\FF$ is viewed as a heap associated to an abelian group as in \eqref{g.h}). Combined with \eqref{0.1} it in fact says that $\Lambda(-,a,b)$ is an additive map from $\FF$ to $A(a)$. Condition \eqref{r.heap} means that $\Lambda(\alpha,a,-): A\lra A$ is a homomorphism of heaps. One can use  \eqref{b.change} to show that also $\Lambda(\alpha, -,b): A\lra A$ is a homomorphism of heaps. Equation \eqref{a.assoc} indicates the associativity of the scalar action of $\FF$ on $\overset{\to}A$. Finally, \eqref{b.change} implies that $\Lambda(\alpha, a,a) = a$, for all $\alpha\in \FF$ and $a\in A$.

\begin{lemma}\label{lem.aff} Let $(A,\Lambda)$ be an affine  space over $\FF$ in the sense of conditions \eqref{a.space}. Then, for all $o\in A$, the abelian group $A(o)$ is a vector space with the multiplication by scalars,
\begin{equation}\label{scalar}
\alpha a = \Lambda(\alpha,o,a).
\end{equation}
Furthermore $A$ is an affine space with the free and transitive action of the vector space $\overset{\lra} A = A(o)$ in which equations \eqref{heap.aff} hold.
\end{lemma}
\begin{proof}
This can be verified by a straightforward calculation. We only mention that the unique vector $\overset{\lra}{ab} \in A(o)$ such that $b= a+\overset{\lra}{ab}$ is given by $\overset{\lra}{ab} = [o,a,b]$.
\end{proof}

Lemma~\ref{lem.aff} indicates that a vector space that appears in the traditional definition of an affine space is a secondary or derived object. In a similar way a linear transformation through which traditionally an affine transformation is defined is an artefact. For,  a homomorphism of affine spaces  $f: A\lra B$ as defined in \eqref{hom.aff} yields a linear transformation
\begin{equation}\label{linearised}
\overset{\to}{f}: \overset{\lra}A \lra \overset{\lra}B, \qquad a\lto [f(a),f(o_A),o_B] = f(a)-f(o_A),
\end{equation}
for any choice $o_A\in A$ and $o_B\in B$ so that $\overset{\lra}A = A(o_A)$ and $\overset{\lra}B = B(o_B)$. One easily checks that $\overset{\to}{f}$ has the required property, 
$$\overset{\to}{f}\left(\overset{\lra}{ab}\right) = \overset{-\!-\!-\!-\!-\!\lra}{f(a)f(b)}.
$$

An associative ring is an abelian group with an associative multiplication distributing over the addition. Exploring the connection between abelian groups and heaps one says that an abelian heap $A$ together with an associative operation $(a,b)\lto ab$ such that, for all $a,b,c,d\in A$,
\begin{equation}\label{truss}
a[b,c,d] = [ab,ac,ad] \quad \& \quad [a,b,c]d = [ad,bd,cd],
\end{equation}
is a {\em truss} \cite{Brz:tru}, \cite{Brz:par}. In other words, in a truss the multiplication distributes over the heap operation. In particular, if $A$ is an affine space with the multiplication $m:A\times A\to A$ that is bi-affine transformation, i.e., for all $a\in A$, both $m(-,a), m(a,-) :A \lra A$ are affine transformations, then $A$ is a truss (since the preservation of heap operations embedded in the definition of an affine transformation yields necessary distributive laws). 

At this point it might be worth stressing  that if $A$ is an abelian group understood as a heap via \eqref{g.h}, the distributive laws \eqref{truss}, which read now
$$
a(b-c+d) = ab-ac+ad \quad \& \quad (a-b+c)d = ad-bd+cd,
$$
do not imply that $a0=0a=0$. Consequently there are typically more trusses than rings (on a given abelian group). For example \cite[Theorem~3.51]{Brz:par} relates all isomorphism classes of commutative trusses on the abelian group of integers $\ZZ$ to orbits of the action of the infinite dihedral group on the set of $2\times 2$ integer idempotent matrices of trace one. Elsewhere \cite{AndBrz:ide} it is shown that there are twenty-three isomorphism classes of trusses on the abelian group $\ZZ_p\oplus \ZZ_p$ as opposed to just eight classes of rings. It might be whorthy of a note that trusses unify rings with {\em braces} \cite{Rum:bra} \cite{CedJes:bra} which play important role in the theory of Jacobson radical rings and the set-theoretic version of the quantum Yang-Baxter equation.

Finally, we would like to mention in passing that despite involving a compatibility between a binary semigroup operation (multiplication) and ternary operation trusses are different from {\em $\mathit{[3,2}]$-rings} introduced in \cite{Cup:rin}, which involve a ternary abelian 3-group structure rather than ternary abelian heap operation.

\section{Heaps of Lie affebras}
The main notion of this note is introduced in the following
\begin{definition}\label{def.Lie}
Let $A$ be an abelian heap with a ternary operation $[-,-,-]$. A {\em Lie bracket} on $A$ is a function
$$
\lb -,-,-\rb  : A\times A\times A\lra A, \qquad (a,b,c)\lto \lb a,b,c\rb ,
$$
such that, for all $a,b\in A$ $\lb a,b,-\rb ,\lb a,-,b\rb , \lb -,a,b\rb :A\lra A$ are heap homomorphisms, and, for all $a,b,c,o\in A$,
\begin{subequations}\label{Lie}
\begin{equation} \label{aa}
\lb a,b,a\rb  = b,
\end{equation}
\begin{equation} \label{as}
\Big[\lb a,b,c\rb ,b,\lb c,b,a\rb \Big] = b,
\end{equation}
\begin{equation} \label{Jacobi}
\lb  \lb a,o,b\rb ,o,c\rb  = \Big[\lb o,o,a\rb , \lb  \lb b,o, c\rb ,o,a\rb , \lb o,o,b\rb , \lb  \lb c,o,a\rb ,o,b\rb , \lb o,o,c\rb \Big].
\end{equation}
\end{subequations}
A heap together with a Lie bracket is called a {\em Lie truss}. 

If $A$ is an affine space and $\lb a,b,-\rb ,\lb a,-,b\rb , \lb -,a,b\rb :A\lra A$ are affine transformations, then $A$ is called a {\em heap of Lie affebras}.
\end{definition}

Condition \eqref{aa} can be seen as the {\em nilpotency} of a Lie bracket. Equation \eqref{as} can be equivalently stated as
$$
\lb c,b,a\rb = \Big[b,\lb a,b,c\rb ,b \Big] .
$$
 In view of the way in which the negation in the retract abelian group $A(b)$ is defined (see equation \eqref{retract}) it  clearly reflects the anti-symmetry  of a Lie bracket. Finally \eqref{Jacobi} is the heap version of the Jacobi identity. It can be equivalently stated as the equality:
$$
\Big[\lb  \lb a,o,b\rb ,o,c\rb  ,\lb o,o,a\rb , \lb  \lb b,o,c\rb ,o,a\rb , \lb o,o,b\rb , \lb  \lb c,o,a\rb ,o,b\rb , \lb o,o,c\rb , o\Big] =o,
$$
for all $a,b,c,o\in A$. In this form its cyclic nature with respect to $a\to b \to c$ is more transparent (and hence it might make the comparison with the usual Jacobi identity for Lie algebras easier).

\begin{remark}\label{rem.strong}
In all examples of Lie trusses and heaps of Lie affebras $(A,\lb -, -,-\rb)$, the Lie bracket satisfies the following {\em strong Jacobi identity}:
\begin{equation} \label{Jacobi.s}
\lb  \lb a,d,b\rb ,e,c\rb  = \Big[\lb d,e,a\rb , \lb  \lb b,d,c\rb ,e,a\rb , \lb d,e,b\rb , \lb  \lb c,d,a\rb ,e,b\rb , \lb d,e,c\rb \Big],
\end{equation}
for all $a,b,c,d,e\in A$. Obviously, \eqref{Jacobi.s} implies \eqref{Jacobi} although it is unlikely that the opposite implication holds in full generality. However, a Lie bracket on a heap can always be modified to produce a bracket which satisfies the strong Jacobi identity; see Corollary~\ref{cor.strong} below.
\end{remark}

\begin{example}\label{ex.Lie.com}
For any truss $A$, the operation
\begin{equation}\label{prod.Lie}
 \lb a,b,c\rb  = [ac,ca,b], \qquad \mbox{for all $a,b,c\in A$},
\end{equation}
is a Lie bracket on the heap $A$ satisfying the strong Jacobi identity \eqref{Jacobi.s}. In particular, $(A, \lb -,-,-\rb )$ is a Lie truss.
\end{example}
\begin{proof}
It might be instructive to check this example explicitly. The bracket is a heap homomorphism on the first argument, by the distributive law of trusses and reshuffling and cancellation rules,
$$
\begin{aligned}
\Big[\lb a,b,c\rb , \lb a',b,c\rb , \lb a'',b,c\rb \Big] &= [ac,ca,b,a'c,ca',b,a''c,ca'',b]\\
&= [ac,a'c, a''c, ca,ca',ca'',b,b,b] = \Big[[a,a',a'']c,c[a,a',a''],b\Big].
\end{aligned}
$$
The heap homomorphism property of the third argument is proven symmetrically, while the one in the middle term is immediate by the cancellation rule. Also properties \eqref{aa} and \eqref{as} follow by the cancellation rule. Next, we start with the left hand side of the strong Jacobi identity
 \eqref{Jacobi.s}
$$
\begin{aligned}
\lb  \lb a,d,b\rb ,e,c\rb  &= \Big[[ab,ba,d]c, c[ab,ba,d], e\Big] = [abc,bac,dc, cab,cba,cd, e],
\end{aligned}
$$
by the truss distributive law. Note that the second and fourth terms in the heap operation on the right hand side of \eqref{Jacobi.s} are obtained by the cyclic permutation of $a,b,c$ on the left hand side, and hence we obtain
$$
\begin{aligned}
\mathrm{RHS}&=[da,ad,e,bca,cba,da, abc,acb,ad, e, db,bd,e, cab,acb,db, bca,bac,bd, e,dc,cd,e].
\end{aligned}
$$
By cancelling identical elements occupying positions of opposite parity and then reshuffling the remainders so that the parity of the positions of elements remains unchanged we obtain:
$$
\begin{aligned}
\mathrm{RHS}&=[cba,cab, abc,bac,dc,cd,e]= [abc,bac,dc, cab,cba,cd, e].
\end{aligned}
$$
This is equal to the left hand side of the strong Jacobi identity as required.
%
\end{proof}

\begin{example}\label{ex.deriv}
A {\em derivation} on a truss $A$ is a heap homomorphism $D:A\lra A$ satisfying the following Leibniz rule
\begin{equation}\label{Leibniz}
D(ab) = \big[D(a)b, ab, aD(b)\big],
\end{equation}
for all $a,b\in A$. The set $\Der(A)$ of all derivations on $A$ is a Lie truss with the pointwise heap operation, for all $D_1,D_2,D_3\in \Der{(A)}$,
\begin{equation}\label{deriv.heap}
\big[D_1, D_2, D_3\big](a) = \big[D_1(a), D_2(a), D_3(a)\big], \qquad \mbox{for all $a\in A$},
\end{equation}
and the Lie bracket
\begin{equation}\label{deriv.Lie}
\blb D_1, D_2, D_3\brb  = \big[D_1D_3, D_3D_1, D_2\big],
\end{equation}
where the composition of derivations is denoted by juxtaposition. Furthermore, the bracket satisfies the strong Jacobi identity \eqref{Jacobi.s}. The detailed study of derivations on trusses is planned to be presented in a forthcoming publication \cite{Brz:der}.
\end{example}
\begin{proof}
Since the form of the bracket \eqref{deriv.Lie} is identical with the form of the bracket \eqref{prod.Lie}, once we can prove that operations \eqref{deriv.heap} and \eqref{deriv.Lie} are well-defined on the set of derivations, the same reasoning as in the proof of Example~\ref{ex.Lie.com} will establish the assertion. Thus, we take any $a,b\in A$ and, using reshuffling and cancellation rules for the heap operation, the Leibniz rule \eqref{Leibniz} as well as the distributive laws of trusses, compute
$$
\begin{aligned}
    \big[D_1, D_2, D_3\big](ab) &=\big[D_1(ab), D_2(ab), D_3(ab)\big] \\
    &=\Big[D_1(a)b, ab, aD_1(b), D_2(a)b, ab, aD_2(b), D_3(a)b, ab, aD_3(b)\Big]\\
    &= \Big[\big[D_1(a)b, D_2(a)b,D_3(a)b\big], [ab,  ab,  ab], \big[aD_1(b), aD_2(b), aD_3(b)\big]\Big]\\
    &= \Big[\big[D_1(a), D_2(a),D_3(a)\big]b, ab, a\big[D_1(b), D_2(b), D_3(b)\big]\Big]\\
    &= \Big[\big[D_1, D_2,D_3\big](a)b, ab, a\big[D_1, D_2, D_3\big](b)\Big].
\end{aligned}
$$
Hence $\big[D_1, D_2,D_3\big]$ is a derivation as required. Next, using the Leibniz rule and the fact that derivations are homomorphisms of heaps, we can compute
$$
\begin{aligned}
\blb D_1, D_2, D_3\brb &(ab)  = \big[D_1D_3(ab), D_3D_1(ab), D_2(ab)\big]\\ 
&= \big[D_1\big(\big[D_3(a)b,ab,aD_3(b)\big]\big), D_3\big(\big[D_1(a)b,ab,aD_1(b)\big]\big), D_2(a)b,ab,aD_2(b)\big]\\
&= \big[D_1D_3(a)b,D_3(a)b , D_3(a)D_1(b), D_1(a)b,ab,aD_1(b),D_1(a)D_3(b),aD_3(b),\\
&~~~~~~~aD_1D_3(b), 
D_3D_1(a)b,D_1(a)b , D_1(a)D_3(b), D_3(a)b,ab,aD_3(b),\\
&~~~~~~~D_3(a)D_1(b),aD_1(b), aD_3D_1(b),D_2(a)b,ab,aD_2(b)\big] .
\end{aligned}
$$
After cancelling identical terms occupying positions of opposite parity and reshuffling remaining elements, so that the parity of their final positions is the same as the original ones, we obtain
$$
\begin{aligned}
\blb D_1, D_2, D_3\brb (ab)  &= \big[D_1D_3(a)b, D_3D_1(a)b, D_2(a)b, ab, aD_1D_3(b), aD_3D_1(b),aD_2(b)\big] \\
&= \Big[\big[D_1D_3(a), D_3D_1(a), D_2(a)\big] b, ab, a\big[D_1D_3(b), D_3D_1(b),D_2(b)\big]\Big]\\
&= \Big[\blb D_1, D_2, D_3\brb (a) b, ab, a\blb D_1, D_2, D_3\brb(b)\Big].
\end{aligned}
$$
In deriving the second equality we have explored the distributive laws of trusses. Thus the operation \eqref{deriv.Lie} is well-defined on the heap of derivations $\Der(A)$.
\end{proof}

Recall from \cite{GraGra:Lie} that a {\em Lie affebra} is an affine space $A$ (over a vector space $\overset{\to}A$) together with an anti-symmetric bi-affine operation
$$
\{-,-\} : A\times A\lra \overset{\to}A, \qquad (a,b)\lto \{a,b\}, 
$$
that satisfies the following Jacobi identity:
\begin{equation}\label{jacobi.aff}
{}_v\{\{a,b,\},c\} + {}_v\{\{b,c,\},a\}+{}_v\{\{c,a,\},b\} = 0,
\end{equation}
where 
$$
{}_v\{-,-\} : \overset{\to}A\times A\lra \overset{\to}A,
$$
denotes the linearisation of the bracket $\{-,-\}$ on the first argument. That is, for all $c\in A$, ${}_v\{-,c\}: a \lto {}_v\{a,c\}$ is the linear transformation corresponding to the affine transformation $\{-,c\}: A \lra \overset{\to}A$, i.e.,
$$
{}_v\Big\{\overset{\lra}{ab},c\Big\} = \overset{-\!-\!-\!-\!-\!-\!-\!-\!-\!-\!\lra}{\{a,c\}\{b,c\}}, \qquad \{a+v,c\} =  \{a,c\}+ {}_v\{v,c\},
$$
for all $a,b\in A$ and $v\in \overset{\to}A$.

The operation $\{-,-\}$ is referred to as a {\em Lie bracket}. The main result of this note is contained in the following
\begin{proposition}\label{prop.affebra}
Let $(A,\{-,-\})$ be a Lie affebra over an $\FF$-vector space $\overset{\to}A$.  View $A$ as a heap with ternary $\FF$-action as in \eqref{heap.aff}. If the field $\FF$ has characteristic different from 2, then $A$ is a heap of Lie affebras with the bracket
\begin{equation}\label{af-heap}
\lb a,b,c\rb  = b+\{a,c\},
\end{equation}
that satisfies the strong Jacobi identity \eqref{Jacobi.s}.

If $(A,\lb -,-,-\rb )$ is a heap of Lie affebras, then, for all $o\in A$, the affine space $A$ over the vector space $\overset{\to} A = A(o)$ is a Lie affebra with the bracket 
\begin{equation}\label{heap-af}
\{-,-\} =  \lb -,o,-\rb .
\end{equation}
\end{proposition}
\begin{proof}
We start with a Lie affebra $(A,\{-,-\})$. If characteristic of $\FF$ is different from 2, then the anti-symmetry of $\{-,-\}$ implies that $\{a,a\} =0$, and hence the property \eqref{aa} for \eqref{af-heap} holds. The property \eqref{as} also follows by the anti-symmetry of $\{-,-\}$.

Since the bracket in a Lie affebra is bi-affine, it is also bi-heap homomorphism, which implies that the bracket \eqref{af-heap} is a heap homomorphism in the first and the third arguments. That it is also a heap homomorphism in the middle one is straightforward. We now check the action preservation property \eqref{hom.aff} in the first argument. Using the fact that $\{-,-\}$ is affine on the first argument we can calculate, for any $\alpha\in \FF$, $a,b,c,d\in A$,
$$
\begin{aligned}
\Lambda(\alpha, \lb a,c,d\rb , \lb b,c,d\rb ) &= \Lambda(\alpha, c+\{a,d\}, c+\{b,d\})\\
&= c+\{a,d\}+ \alpha \ \overset{-\!-\!-\!-\!-\!-\!-\!-\!-\!-\!-\!-\!-\!-\!-\!-\!-\!-\!-\!-\!-\!-\!-\!\lra}{(c+\{a,d\})( c+\{b,d\})} \\
&=  c+\{a,d\} + \alpha \ (\{b,d\} -\{a,d\} ) \\
&= c +\{a,d\} +\alpha \ {}_v\{\overset{\lra}{ab}, d\}   = c+\{a,d\} + {}_v\{\alpha \overset{\lra}{ab}, d\}  \\
&=c+ \{a+\alpha \overset{\lra}{ab}, d\}  = \lb \Lambda(\alpha, a,b),c,d\rb .
\end{aligned}
$$
The action-preservation in the third argument follows by the fact that $\{-,-\}$ is affine in the second argument. For the middle one, we can compute
$$
\begin{aligned}
\Lambda(\alpha, \lb a,b,d\rb , \lb a,c,d\rb ) &= b+ \{a,d\} + \alpha \ \overset{-\!-\!-\!-\!-\!-\!-\!-\!-\!-\!-\!-\!-\!-\!-\!-\!-\!-\!-\!-\!-\!-\!-\!\lra}{(b+\{a,d\})(c+\{a,d\})}\\
&=  b +  \alpha \ \overset{\lra}{bc} + \{a,d\} = \Lambda(\alpha, b,c) + \{a,d\} =
\lb a, \Lambda(\alpha, b,c), d\rb ,
\end{aligned}
$$
as required.

We start with the left-hand side of the strong Jacobi identity \eqref{Jacobi.s},
$$
\begin{aligned}
\mathrm{LHS} &= \lb  \lb a,d,b\rb ,e,c\rb  = \lb d+\{a,b\}, e,c\rb  = e+\{d+\{a,b\},c\}  = e+\{d,c\} + {}_v\{\{a,b\},c\}
\end{aligned}
$$
Using the cyclic nature of the Jacobi identity it is easily found that
$$
\begin{aligned}
\mathrm{RHS}&=e+\{d,c\}+\overset{-\!-\!-\!-\!-\!-\!-\!-\!-\!-\!-\!-\!-\!-\!-\!-\!-\!-\!-\!-\!-\!-\!-\!-\!-\!-\!-\!-\!-\!-\!-\!-\!-\!-\!-\!-\!-\!-\!-\!\lra}{(e+\{d,a\} +{}_v\{\{b,c\},a\})(e+\{d,a\})}  + \overset{-\!-\!-\!-\!-\!-\!-\!-\!-\!-\!-\!-\!-\!-\!-\!-\!-\!-\!-\!-\!-\!-\!-\!-\!-\!-\!-\!-\!-\!-\!-\!-\!-\!-\!-\!-\!-\!-\!-\!\lra}{(e+\{d,b\} +{}_v\{\{c,a\},b\})(e+\{d,b\})} \\
&= e+\{d,c\}  -{}_v\{\{b,c\},a\}  - {}_v\{\{c,a\},b\}  = \mathrm{LHS},
\end{aligned}
$$
by the Jacobi identity \eqref{jacobi.aff}.

Conversely, assume that $(A,\lb -,-,-\rb )$ is a heap of Lie affebras and choose to view $A$ as an affine space over the vector space $\overset{\to}A = A(o)$. Since $\lb -,-,-\rb $ is affine on the first and third arguments, the bracket $\{-,-\}$ given by \eqref{heap-af} is bi-affine. The antisymmetry of $\{-,-\}$ follows by \eqref{as} (note that $o$ is the zero vector in $\overset{\to}A$ and $[a,o,b] = a+b$ in $\overset{\to}A$). In view of \eqref{linearised}
$$
{}_v\{a,b\} = \{a,b\} - \{o,b\} = \lb a,o,b\rb  - \lb o,o,b\rb .
$$
Making use of the fact that $[a,b,c] = a-b+c$ in $A(o)$, the Jacobi identity \eqref{Jacobi} can be rewritten as:
$$
\lb \lb a,o,b\rb ,o,c\rb  = \lb o,o,a\rb  - \lb \lb b,o,c\rb ,o,a\rb  +\lb o,o,b\rb  - \lb \lb c,o,a\rb ,o,b\rb  + \lb o,o,c\rb ,
$$
or, equivalently,
$$
\{ \{a,b\},c\} - \{o,c\} + \{\{b,c\},a\} - \{o,a\} +\{\{c,a\},b\} - \{o,b\} =o,
$$
that is,
$$
{}_v\{\{a,b\},c\}+ {}_v\{\{b,c\},a\}  +{}_v\{\{c,a\},b\}  =o,
$$ 
as required for a Lie affebra. 
\end{proof}

Proposition~\ref{prop.affebra} immediately implies that if $A$ is a Lie algebra (over a field of characteristic different from  2), then it is also a heap of Lie affebras with the ternary Lie bracket \eqref{af-heap}. More generally, if $A$ is a Lie ring, then  it is a Lie truss, with the ternary Lie bracket given by formula \eqref{af-heap}. Conversely, 
 every Lie truss leads to a family of Lie rings.

\begin{proposition}\label{prop.ring}
If $(A,\lb -,-,-\rb )$ is a Lie truss (resp.\ heap of Lie affebras), then, for all $o\in A$, the retract $A(o)$ is a Lie ring (resp.\ Lie algebra) with the bracket:
\begin{equation}\label{Lie.bracket}
{}[a,b] = \lb a,o,b\rb  - \lb a,o,o\rb  - \lb o,o,b\rb  .
\end{equation}
\end{proposition}
\begin{proof}
Recall that from the point of view of the abelian group structure $A(o)$, the heap operation on $A$ has the form $[a,b,c] = a-b+c$. First we prove that the bracket \eqref{Lie.bracket} is additive in both arguments. For all $a,b,c\in A$,
$$
\begin{aligned}
{}[a+b,c] &= \lb a-o+b,o,c\rb  - \lb a-o+b,o,o\rb  - \lb o,o,c\rb  \\
&= \lb a,o,c\rb  - \lb o,o,c\rb  + \lb b,o,c\rb  - \lb a,o,o\rb  + \lb o,o,o\rb -\lb b,o,o\rb  - \lb o,o,c\rb  \\
&= [a,c] +[b,c],
\end{aligned}
$$
since $\lb -,-,-\rb $ is a heap operation on the first argument and $\lb o,o,o\rb  =o$ by \eqref{aa}. The additivity in the second argument is proven in a similar way. In case of the heap of Lie affebras the linearity of \eqref{Lie.bracket} follows by the preservation of the $\FF$-action $\Lambda$ on $A$ by the first and third entries in $\lb -,-,-\rb $. The bracket \eqref{Lie.bracket} is antisymmetric by \eqref{as}. 

Written in terms of addition in $A(o)$, Jacobi identity \eqref{Jacobi}  reads:
\begin{equation}\label{J.1}
\lb \lb a,o,b\rb ,o,c\rb  + \lb \lb b,o,c\rb ,o,a\rb  + \lb \lb c,o,a\rb ,o,b\rb  = \lb o,o,a\rb + \lb o,o,b\rb  +\lb o,o,c\rb .
\end{equation}
Consequently and in view of \eqref{aa},
\begin{equation}\label{J.2}
\lb \lb a,o,b\rb ,o,o\rb  + \lb \lb b,o,o\rb ,o,a\rb  + \lb \lb o,o,a\rb ,o,b\rb  = \lb o,o,a\rb + \lb o,o,b\rb ,
\end{equation}
and 
\begin{equation}\label{J.3}
\lb \lb a,o,o\rb ,o,o\rb  +  \lb \lb o,o,a\rb ,o,o\rb  = o.
\end{equation}
On the other hand,
$$
\begin{aligned}
{}[[a,b],c] =&\; \lb \lb a,o,b\rb ,o,c\rb  - \lb \lb a,o,b\rb ,o,o\rb   - \lb \lb a,o,o\rb ,o,c\rb   - \lb \lb o,o,b\rb ,o,c\rb \\
& + \lb \lb a,o,o\rb ,o,o\rb   + \lb \lb o,o,b\rb ,o,o\rb  +\lb o,o,c\rb .
\end{aligned}
$$
Therefore,
$$
\begin{aligned}
{}[[a,b],c] + \mbox{cycl.} =&
\lb \lb a,o,b\rb ,o,c\rb  - \lb \lb a,o,b\rb ,o,o\rb   - \lb \lb a,o,o\rb ,o,c\rb   - \lb \lb o,o,b\rb ,o,c\rb \\
& + \lb \lb a,o,o\rb ,o,o\rb   + \lb \lb o,o,b\rb ,o,o\rb  +\lb o,o,c\rb \\
&+ \lb \lb b,o,c\rb ,o,a\rb  - \lb \lb b,o,c\rb ,o,o\rb   - \lb \lb b,o,o\rb ,o,a\rb   - \lb \lb o,o,c\rb ,o,a\rb \\
& + \lb \lb b,o,o\rb ,o,o\rb   + \lb \lb o,o,c\rb ,o,o\rb  +\lb o,o,a\rb \\
&+ \lb \lb c,o,a\rb ,o,b\rb  - \lb \lb c,o,a\rb ,o,o\rb   - \lb \lb c,o,o\rb ,o,b\rb   - \lb \lb o,o,a\rb ,o,b\rb \\
& + \lb \lb c,o,o\rb ,o,o\rb   + \lb \lb o,o,a\rb ,o,o\rb  +\lb o,o,b\rb \\
=&\; \lb o,o,a\rb + \lb o,o,b\rb  +\lb o,o,c\rb  - (\lb o,o,a\rb + \lb o,o,b\rb ) - (\lb o,o,b\rb + \lb o,o,c\rb )\\
& - (\lb o,o,c\rb + \lb o,o,a\rb ) + \lb o,o,a\rb  + \lb o,o,b\rb  + \lb o,o,c\rb   = o ,
\end{aligned}
$$
by \eqref{J.1}--\eqref{J.3}. Thus we conclude that $A(o)$ with the bracket \eqref{Lie.bracket} is a Lie ring (resp.\ Lie algebra).
\end{proof}

Combining Proposition~\ref{prop.affebra} with Proposition~\ref{prop.ring} we thus obtain the following way of generating Lie ternary brackets that satisfy the strong version of the Jacobi identity \eqref{Jacobi.s}.

\begin{corollary}\label{cor.strong}
Let $(A, \lb -,-,-\rb)$ be a Lie truss. Then, for all $o\in A$, the ternary operation given by
\begin{equation}\label{strong}
    \lb a,b,c\rb_o = \big[\lb a,o,c\rb , \lb a,o,o\rb, o, \lb o,o,c\rb ,b\big], 
\end{equation}
for all $a,b,c\in A$ is a Lie bracket on the heap $A$ that satisfies the strong Jacobi identity \eqref{Jacobi.s}.
\end{corollary}
\begin{proof}
By Proposition~\ref{prop.ring}, $A(o)$ is a Lie ring with the bracket
$$
{}[a,c] = \lb a,o,c\rb  - \lb a,o,o\rb -\lb o,o,c\rb = \big[\lb a,o,c\rb , \lb a,o,o\rb, o, \lb o,o,c\rb ,o\big].
$$
Since 
$$[a,b,c] = [a,o, [o, b,o],o,c] = a-b+c
$$
in $A(o)$, by the arguments of Proposition~\ref{prop.affebra} made explicit after its proof, $A$ is a Lie truss with the bracket
$$
\lb a,b,c\rb_o = [a,c]+b = \big[[a,c],o,b\big] = \big[\lb a,o,c\rb , \lb a,o,o\rb, o, \lb o,o,c\rb ,b\big]
$$
that satisfies the strong Jacobi identity \eqref{Jacobi.s}.
\end{proof}

\section*{Acknowledgements}
I would like to thank Andrew Bruce for drawing my attention to the papers by K.\ Grabowska, J.\ Grabowski \& P.\ Urba\'nski on Lie affebras.
The research is partially supported by the National Science Centre, Poland, grant no. 2019/35/B/ST1/01115.


\begin{thebibliography}{99}
\bibitem{ncat:aff} {\em Affine spaces}, nLab: http://nlab-pages.s3.us-east-2.amazonaws.com/nlab/show/affine+space
\bibitem{AndBrz:ide} R.R.\ Andruszkiewicz, T.\ Brzezi\'nski \& B.\ Rybo\l owicz, Ideal ring extensions and trusses, {\em J.\ Algebra}  (2022). \href{https://doi.org/10.1016/j.jalgebra.2022.01.038}{https://doi.org/10.1016/j.jalgebra.2022.01.038}.
\bibitem{Bae:ein} R.\ Baer, Zur Einf\"uhrung des Scharbegriffs, {\em J.\ Reine Angew.\ Math.} {\bf160} (1929), 199--207.
\bibitem{Ben:fib} S.\ Benenti, Fibr\'es affines canoniques et m\'ecanique newtonienne. {\em Action hamiltoniennes de groupes. Troisi\`eme th\'eor\`eme de Lie (Lyon, 1986)}, 13--37, {\em Travaux en Cours}, {\bf 27}, Hermann, Paris, 1988.
\bibitem{BreBrz:hea} S.\ Breaz, T.\ Brzezi\'nski, B.\ Rybo\l owicz \& P.\ Saracco, Heaps of modules and affine spaces, arXiv:2203.07268 (2022).
\bibitem{Brz:tru} T.\ Brzezi\'nski, Trusses: Between braces and rings, {\em Trans.\ Amer.\ Math.\ Soc.} {\bf 372} (2019), 4149--4176.
\bibitem{Brz:par} T.\ Brzezi\'nski, Trusses: Paragons, ideals and modules, {\em J.\ Pure Appl.\ Algebra} {\bf 224} (2020), 106258.
\bibitem{Brz:der} T.\ Brzezi\'nski, Derivations on trusses, {\em in preparation} (2022).
\bibitem{Bru:sem} A.J.\ Bruce, {Semiheaps and ternary algebras in quantum mechanics revisited}, {\em Universe} {\bf 8} (2022) No.\ 1 Art.\ 56.
\bibitem{CedJes:bra} F.\ Ced\'o, E.\ Jespers \& J.\ Okni\'nski,  Braces and the Yang-Baxter equation, {\em Commun.\ Math.\ Phys.} {\bf 327} (2014), 101--116.
\bibitem{Cup:rin}  G.\ \v Cupona,  On [m,n]-rings, {\em Bull.\ Soc.\ Math.\ Phys.\ Mac\'edoine} {\bf 16} (1965), 5--9 (in Macedonian).
\bibitem{DupWer:str} S.\ Duplij \& W. Werner, Structure of unital 3-fields, {\em Math.\ Semesterber.} {\bf 68} (2021), 27--53.
\bibitem{Fil:Lie} V.T.\ Filippov, $n$-Lie algebras, {\em Sibirsk.\ Math.\ Zh.} {\bf 26} (1985), 126--140.
\bibitem{GraGra:Lie}  K.\ Grabowska, J.\ Grabowski \& P.\ Urba\'nski,  Lie brackets on affine bundles, {\em Ann.\ Global Anal.\ Geom.} {\bf 24} (2003), 101--130.
\bibitem{GraGra:AV} K.\ Grabowska, J.\ Grabowski \& P.\ Urba\'nski, AV-differential geometry: Poisson and Jacobi structures, {\em J.\ Geom.\ Phys.} {\bf 52} (2004),  398--446.
\bibitem{GraGra:fra} K.\ Grabowska, J.\ Grabowski \& P.\ Urba\'nski, Frame-independent mechanics: geometry on affine bundles, {\em Travaux mathematiques. Fasc. XVI}, 107--120, {\em Trav.\ Math.}, {\bf 16}, Univ.\ Luxemb., Luxembourg, 2005.
\bibitem{Jac:Lie} N.\ Jacobson, Lie and Jordan triple systems, {\em Amer.\ J.\ Math.}, {\bf 71} (1949), 149--170.
\bibitem{Ker:ter0}  R.\ Kerner, Ternary algebraic structures and their application in physics, {arXiv:math-ph/0011023} (2000).
\bibitem{Ker:ter}  R.\ Kerner, Ternary and non-associative structures, {\em Int.\ J.\ Geom.\ Methods Mod.\ Phys.} {\bf 5} (2008), 1265--1294.
 \bibitem{MasVig:non}  E.\ Massa, S.\ Vignolo \& D.\ Bruno,  Non-holonomic Lagrangian and Hamiltonian mechanics: an intrinsic approach, {\em J.\ Phys.\ A} {\bf 35} (2002), 6713--6742. 
 \bibitem{Pru:the} H.\ Pr\"ufer, {\em Theorie der Abelschen Gruppen. I.\ Grundeigenschaften}, {Math.\ Z.} {\bf 20} (1924), 165--187.
 \bibitem{Rum:bra} W.\ Rump,  Braces, radical rings, and the quantum Yang-Baxter equation,  
{\em J.\ Algebra} {\bf 307} (2007), 153--170.
 \bibitem{Tul:fra} W. Tulczyjew, Frame independence of analytical mechanics, {\em Atti Accad.\ Sci.\ Torino Cl.\ Sci.\ Fis.\ Mat.\ Natur.} {\bf  119 } (1985), 273--279.
 \bibitem{Urb:aff} P.\ Urba\'nski, Affine Poisson structures in analytical mechanics, {\em Quantization and infinite-dimensional systems (Bia\l owie\.za, 1993)}, 123--129, Plenum, New York, 1994.
 \bibitem{Wei:uni} A.\ Weinstein, A universal phase space for particles in Yang-Mills fields, {\em Lett.\ Math.\ Phys.} {\bf 2} (1977/78), 417--420.
 \end{thebibliography}
\end{document}